\documentclass[11pt, reqno]{amsart}

\usepackage{amsmath,amssymb,amsthm, epsfig}
\usepackage{hyperref}
\hypersetup{colorlinks=true, linkcolor=blue, citecolor=red, linktoc=page, pdftitle={Rectifiability of free boundaries}, pdfauthor={R. Teymurazyan}}
\usepackage{ulem}
\usepackage[mathscr]{euscript}
\usepackage[latin1]{inputenc}

\usepackage{xcolor}
\usepackage{hyperref}
\usepackage[nameinlink]{cleveref}
\hypersetup{
	colorlinks,
	linkcolor={blue!50!black},
	citecolor={blue!50!black},
	urlcolor={blue!50!black}
}

\usepackage{esint}
\usepackage{enumerate}
\usepackage{mathrsfs}
\usepackage[usenames,dvipsnames]{pstricks}
\usepackage{pst-grad} 
\usepackage{pst-plot} 

\newtheorem{theorem}{Theorem}
\newtheorem{definition}{Definition}

\newtheorem{lemma}{Lemma}
\newtheorem{proposition}{Proposition}
\newtheorem{corollary}{Corollary}

\date{}
\numberwithin{equation}{section}
\numberwithin{theorem}{section}
\numberwithin{lemma}{section}
\numberwithin{corollary}{section}
\numberwithin{remark}{section} 
\numberwithin{proposition}{section}
\numberwithin{definition}{section}

\newcommand{\R}{\mathbb{R}}

\newcommand{\dist}{\operatorname{dist}}
\newcommand{\loc}{\operatorname{loc}}
\def \H {\mathcal{H}^{n-1}}

\newcommand{\reduced}{\operatorname{red}}
\newcommand{\interior}{\operatorname{int}}

\begin{document}





\title[Rectifiability of free boundaries]{Rectifiability of free boundaries in singular diffusion problems}

\author[R. Teymurazyan]{Rafayel Teymurazyan}
	\address{Applied Mathematics and Computational Sciences Program (AMCS), Computer, Electrical and Mathematical Sciences and Engineering Division (CEMSE), King Abdullah University of Science and Technology (KAUST), Thuwal, 23955-6900, Kingdom of Saudi Arabia}{}
	\email{rafayel.teymurazyan@kaust.edu.sa}

\begin{abstract}
For minimizers of a degenerate diffusion functional with a singular reaction term, we prove that the free boundary is $(n-1)$-rectifiable. The argument relies on a suitable integrability property, derived from a pointwise gradient estimate, combined with a Hausdorff dimension estimate for a portion of the zero set.
\end{abstract}

\maketitle

\section{Introduction}
We study the rectifiability of the free boundary of minimizers of
\begin{equation}\label{mimization functional}
J_\Omega(v):=\int_{\Omega}\left(\frac{|Dv|^p}{p} + v^\gamma \right)\,dx
\end{equation}
over
\begin{equation}\label{minimization set}
    \mathbb{K}:=\left\{v\in W^{1,p}(\Omega):\,\,v-g\in W^{1,p}_0(\Omega), \ v\geq 0 \mbox{ a.e. in } \Omega\right\}.
\end{equation}
Here $\Omega\subset\R^n$ is a bounded domain, $\gamma\in(0,1)$, $p\ge2$ and $g\in W^{1,p}(\Omega)$ is a non-negative function. As shown in \cite[Proposition 2.1]{ATU25}, minimizers satisfy
\begin{equation}\label{EL}
\Delta_p u = \gamma u^{\gamma-1}
\quad\text{in}\quad\{u>0\}\cap\Omega,
\end{equation}
where $\Delta_p u:=\mathrm{div}(|Du|^{p-2}Du)$ is the $p$-Laplacian. Since $\gamma<1$, the right-hand side in \eqref{EL} blows up as $u\to0^+$, making the study of the behavior of minimizers near the free boundary $\partial\{u>0\}$ highly nontrivial.

The problem was studied by Phillips in the linear setting ($p=2$) in \cite{P83,Ph83} (see also \cite{AP86,AT13}), proving that minimizers are locally of the class $C^{1,\frac{\gamma}{2-\gamma}}$, and that the free boundary has locally finite $(n-1)$-dimensional Hausdorff measure and is rectifiable. As the approach in \cite{P83,Ph83} relies on the linearity of the operator, it does not extend directly to $p>2$. Nevertheless, using geometric tangential analysis, several of these results were established in the degenerate setting in \cite{ATV23,ATU25}.

Despite recent advances, the rectifiability of the free boundary remained open in the nonlinear setting. In this note, we close that gap by proving that the free boundary has a measure-theoretic outward normal at $\H$-almost every point. More precisely, our main result shows that if $u$ is a minimizer of \eqref{mimization functional}, then locally,
\begin{equation}\label{reduced}
    \H(\partial\{u>0\}\setminus\partial_{\reduced}\{u>0\})=0.
\end{equation}
The proof relies on two main ingredients: the integrability of a nonlinear transformation of $u$, derived from the pointwise gradient estimate established in \cite{ATU25}, and a Hausdorff dimension estimate for a portion of the zero set. Since $\{u>0\}$ has locally finite perimeter, its reduced boundary is locally $(n-1)$-rectifiable, \cite[Corollary 16.1]{M12}. Hence, \eqref{reduced} establishes the same property for the free boundary $\partial\{u>0\}$.

The limiting cases $\gamma=0$ and $\gamma=1$ correspond to the cavity and obstacle problems, respectively. For $\gamma=0$, rectifiability was obtained in \cite[Corollary 9.2]{DP05}. For $\gamma=1$, the free boundary has locally finite $(n-1)$-dimensional Hausdorff measure \cite{LS03,CLRT14} and is rectifiable for $p=2$, \cite{C98}. For $p>2$, the free boundary of the obstacle problem is countably $(n-1)$-rectifiable under a concavity condition that excludes constant obstacles, \cite[Theorem 1.8]{FKR17}.

\vspace{0.2 cm}

The functional \eqref{mimization functional} models the competition between nonlinear diffusion and a singular reaction term. Such structures arise in reaction-diffusion models with thresholding effects, where the unknown represents quantities such as density, pressure, or concentration. In these settings, the positivity set corresponds to an active region, while the zero phase represents inactive zones, and the free boundary describes the interface between them. 

\vspace{0.2cm}

This note is organized as follows. After listing the main notations and some known results in \Cref{s2}, in \Cref{s3} we prove an integrability property, \Cref{integrability lemma}. In \Cref{s4}, we show that portions of the zero set with vanishing Lebesgue measure are negligible in a certain Hausdorff dimension, \Cref{Hausdorff local}. Finally, \Cref{s5} extends \eqref{reduced} into the degenerate setting, \Cref{Hausdorff reduced}. 

\section{Preliminaries}\label{s2}
In this section, we gather the known results on minimizers of \eqref{mimization functional}-\eqref{minimization set} that will be used in this work. We start by introducing some notations and recalling several definitions that are used throughout the work.
\subsection{Definitions}
The ball of radius $r$ centered at $x$ is denoted by $B_r(x)$, and we write $B_r:=B_r(0)$. Also $J_r:=J_{B_r}$, where $J_{B_r}$ is defined by \eqref{mimization functional}. $|E|$ is the $n$-dimensional Lebesgue measure of the set $E$, and $\mathcal{H}^d(E)$ denotes its $d$-dimensional Hausdorff measure.
\begin{definition}\label{reduced free boundary}
    A point $x\in B_1$ is said to be on the \textit{reduced} free boundary of $\{u>0\}$ in $B_1$, and we write $x\in\partial_{\reduced}\{u>0\}$, if there exists a unique unit vector $\nu(x)$ such that
    $$
    \int_{B_r(x)}\left|\chi_{\{u>0\}}-\chi_{ \{ y \, : \, (y-x)\cdot\nu(x)<0 \} } \right| \, dy = o(r^n),
    $$
    as $r\to0$.   
\end{definition}
By De Giorgi's structure theorem, the reduced boundary of a set of locally finite perimeter is a locally $\H-$rectifiable set in $\R^n$ (see, for example, \cite[Corollary 16.1]{M12}). Thus, the reduced free boundary is the subset of the free boundary, where $\{u>0\}$ has a tangent plane in measure with normal $\nu(x)$.
\begin{definition}\label{localminimizer}
A nonnegative function $u\in W^{1,p}(\Omega)$ is a local minimizer of \eqref{mimization functional} in $B_r$, if
$$
J_r(u)\le J_r(v),
$$
for all $v\in W^{1,p}(B_r)$ such that $v-u\in W_0^{1,p}(B_r)$.
\end{definition}
\subsection{Known results}
As it is established in \cite[Theorem 2.1]{ATV23}, minimizers of \eqref{mimization functional}-\eqref{minimization set} exist, $0\le u\le\|g\|_{L^\infty(\Omega)}$, and satisfy the Euler-Lagrange equation \eqref{EL}, \cite[Proposition 2.1]{ATU25}. Furthermore, \cite[Theorem 3.1]{ATV23}, we have
\begin{equation}\label{local regularity}
    \|u\|_{C^{1,\beta}(\Omega')}\le C,
\end{equation}
for every $\Omega'\subset\Omega$, and $C>0$ depends only on $\dist(\Omega',\partial\Omega)$, $\|u\|_{L^\infty(\Omega)}$, $p$, $\gamma$ and $n$. Here $\beta:=\left\{\alpha_p^-,\frac{\gamma}{p-\gamma}\right\}$, $\alpha_p$ is the optimal (unknown for $n\ge3$) regularity exponent for the gradient of $p$-harmonic functions, and $c^-$ means any real number less than $c$. Moreover, as shown in \cite[Theorem 4.1 and Theorem 6.1]{ATV23}, across the free boundary points, one has
\begin{equation}\label{growth}
    u\le Cr^{\frac{p}{p-\gamma}}\,\,\,\text{in}\,\,\,B_r(y),
\end{equation}
for any $y\in\partial\{u>0\}\cap\Omega'$, and for any $r\in(0,r_0)$, with some universal constants $C>0$ and $r_0>0$. For $x_0\in\overline{\{u>0\}}$, we also have the following non-degeneracy for minimizers, \cite[Proposition 4.1]{ATU25}
\begin{equation}\label{nondegeneracy}
    \sup_{\partial B_r(x_0)}u\ge cr^{\frac{p}{p-\gamma}},
\end{equation}
for any $r<\dist(x_0,\partial \Omega)$. Here $c>0$ depends only on $p$, $\gamma$ and $n$. Additionally, the following positive density result holds, \cite[Corollary 4.2]{ATU25},
\begin{equation}\label{positive density}
    \frac{|\{u>0\}\cap B_r(x_0)|}{|B_r(x_0)|}\ge\tau,
\end{equation}
where $x_0\in\Omega'\cap\partial\{u>0\}$, $\Omega'\Subset\Omega$, $r>0$ is small enough, and $\tau \in(0,1)$ depends only on $\Omega'$, $p$, $\gamma$ and $n$. Furthermore, we know that $\partial\{u>0\}$ has locally finite $\H$-measure, \cite[Theorem 6.1]{ATU25}, and $\{u>0\}$ has locally finite perimeter in $\Omega$, \cite[Remark 6.1]{ATU25} (see also \cite[Theorem 4.5.11]{F69} and \cite{M12,Ph83}). 

The result below is from \cite[Theorem 3.1]{ATU25}. It provides a pointwise gradient estimate which is vital in our analysis.
\begin{theorem}\label{magic theorem}
If $u$ is a local minimizer of \eqref{mimization functional} in $B_1$, then there exists $C>0$, depending only on $p$, $\gamma$ and $n$, such that
$$
|Du(x)|^p\le C\|u\|_{L^\infty(B_1)}^{p-\gamma}\, u^{\gamma}(x),\quad x\in B_{\frac{1}{2}}.
$$
\end{theorem}
The next result is from \cite[Proposition 3.3]{LQT15} (see also \cite[Theorem 2.8]{Ph83}).
\begin{proposition}\label{equation in the sense of distributions}
    If $\gamma\in(0,1)$, $p>1$ and $u$ is a minimizer of \eqref{mimization functional}, then
    $$
    \Delta_pu=\gamma u^{\gamma-1}\chi_{\{u>0\}}\,\,\,\text{in}\,\,\,\Omega
    $$
    in the sense of distributions.
\end{proposition}
We close this section by recalling \cite[Proposition 5.1]{ATU25}, revealing that in a universally flat regime, the function $u^{\frac{p-\gamma}{p}}$ is $p$-subharmonic.
\begin{proposition}\label{subsol}
If $u$ is a local minimizer of \eqref{mimization functional} in $B_1$, then there exists $\varepsilon_\star>0$, depending only on $p$, $\gamma$ and $n$, such that 
$$
\Delta_p\left(u^{\frac{p-\gamma}{p}}\right)\ge0 \quad \mbox{in} \quad \left\{0<u\le\varepsilon_\star^{\frac{p}{p-\gamma}}\right\} \cap B_{\frac{1}{2}}.
$$
\end{proposition}

\section{An integrability property}\label{s3}
In this section, we prove integrability of the $p$-Laplacian of a nonlinear transform of minimizers, which is used later to obtain our main result. Using \Cref{magic theorem}, we extend the corresponding result from \cite[Corollary 2.7]{Ph83} for $p=2$ to the nonlinear setting with $p\ge2$.
\begin{lemma}\label{integrability lemma}
    If $u$ is a minimizer of \eqref{mimization functional}, and $x_0\in\partial\{u>0\}$, then
    $$
    \Delta_p(u^{\frac{p-\gamma}{p}})\in L^1\left(\{u>0\}\cap B_r(x_0)\right),
    $$
    for $r>0$ small.
\end{lemma}
\begin{proof}
    Let $\varphi\in C^\infty(\R)$ be such that $\varphi'\ge0$, and
    \begin{equation*}
    \varphi(t):=
    \begin{cases}
    0 \quad \textrm{if} \quad t\le\frac{1}{2},\\
    t \quad \textrm{if} \quad t\ge 1.
    \end{cases}
    \end{equation*}
    For each $0<\delta<\varepsilon$, set
    \begin{equation*}
    \varphi_\delta(t):=\delta\varphi\left(\frac{t}{\delta}\right)\,\,\,\text{and}\,\,\,u_\varepsilon:=\min\left\{ u^\frac{1}{\alpha},\varepsilon\right\},
    \end{equation*}
    where $\alpha:=\frac{p}{p-\gamma}$. Since $\varphi_\delta'\ge0$ and
    \begin{equation*}
        D\varphi_\delta(u_\varepsilon)=0\,\,\,\text{on}\,\,\,\{u\ge\varepsilon^\alpha\},
    \end{equation*}
    then
    \begin{align}\label{I is positive}
        I&:=\int_{B_r(x_0)}|D(u^{\frac{1}{\alpha}})|^{p-2}D(u^{\frac{1}{\alpha}})\cdot D\varphi_\delta(u_\varepsilon)\,dx\nonumber\\
        &=\int_{\{u<\varepsilon^\alpha\}\cap B_r(x_0)}|D(u^{\frac{1}{\alpha}})|^p\varphi'_\delta(u_\varepsilon)\,dx\ge0
    \end{align} 
    On the other hand, taking $\delta=\frac{\varepsilon}{2}$ and integrating by parts, we get
    \begin{align}\label{integration by parts}
        I&=\int_{\partial B_r(x_0)}\varphi_\delta(u_\varepsilon)|D(u^{\frac{1}{\alpha}})|^{p-2}D(u^{\frac{1}{\alpha}})\cdot\nu\,dx\nonumber\\
        &\quad-\int_{B_r(x_0)}\varphi_\delta(u_\varepsilon)\Delta_p(u^{\frac{1}{\alpha}})\,dx:=I_1-I_2.
    \end{align}
    Since 
    $$
     D(u^{\frac{1}{\alpha}})=\frac{1}{\alpha}u^{-\frac{\gamma}{p}}Du,
    $$
    then employing \Cref{magic theorem}, we obtain
    $$
    |D(u^{\frac{1}{\alpha}})|^{p-2}D(u^{\frac{1}{\alpha}})\cdot\nu\le\frac{1}{\alpha^{p-1}}\left[u^{-\frac{\gamma}{p}}|Du|\right]^{p-1}\le\frac{C}{\alpha^{p-1}}\|u\|^{p-\gamma}_{L^\infty(B_1)}.
    $$
    Hence, the integrand in $I_1$ is bounded and therefore
    \begin{equation}\label{I1 is bounded}
        |I_1|\le C,
    \end{equation}
    where $C>0$ is independent of $\varepsilon>0$.
    Combining \eqref{I is positive}-\eqref{I1 is bounded}, we get
    \begin{equation}\label{I_2 is bounded}
        I_2\le I_1\le C.
    \end{equation}
    Observe now that
    \begin{equation}\label{decomposing I2}
        I_2=\int_{\{u\le\varepsilon^\alpha\}\cap B_r(x_0)}\varphi_\delta(u)\Delta_p(u^{\frac{1}{\alpha}})\,dx+\int_{\{u>\varepsilon^\alpha\}\cap B_r(x_0)}\Delta_p(u^{\frac{1}{\alpha}})\,dx.
    \end{equation}
    For $\varepsilon>0$ small enough, according to \Cref{subsol},
    $$
    \Delta_p(u^{\frac{1}{\alpha}})\ge0\,\,\,\text{in}\,\,\, \{0<u\le\varepsilon^\alpha\},
    $$
    hence, the first term in the right-hand side of \eqref{decomposing I2} is nonnegative, which combined with \eqref{I_2 is bounded}, yields
    $$
    \int_{\{u>\varepsilon^\alpha\}\cap B_r(x_0)}\Delta_p(u^{\frac{1}{\alpha}})\,dx\le I_2\le C.
    $$
    Letting $\varepsilon\to0^+$ in the last inequality, we obtain the desired result.
\end{proof}
\section{Negiligibility of a certain zero set}\label{s4}
In this section, we show that part of the free boundary away from the reduced boundary is negligible in a certain Hausdorff dimension. We conclude it from the following result, which extends the corresponding statement for $p=2$ to any $p\ge2$. Its proof follows the ideas of \cite[Lemma 3.1]{Ph83}; the additional terms arising from the nonlinearity are controlled using the pointwise gradient estimate from \Cref{magic theorem}. We provide the details below.
\begin{lemma}\label{Hausdorff local}
    If $u$ is a minimizer of \eqref{mimization functional}, $B_{3r}\subset\Omega$ and $|\{u=0\}\cap B_{3r}|=0$, then
    \begin{equation}\label{step1}
        \mathcal{H}^{n-\frac{2-\gamma}{p-\gamma}p}\left(\{u=0\}\cap B_{r}\right)=0.
    \end{equation}
\end{lemma}
\begin{proof} We divide the proof into two steps.

    \textsc{Step 1.} First, we show that
    \begin{equation}\label{new integrability}
        u^{\gamma-2}\in L^1(B_{2r}).
    \end{equation}
    Let $\xi\in C_0^\infty(B_{3r})$, $\xi\ge0$ and $\xi\equiv1$ on $B_{2r}$. Then for any $\varepsilon\ge0$, the function $u+\varepsilon\xi\in\mathbb{K}$. Setting 
    $$
    I(\varepsilon):=J(u+\varepsilon\xi),
    $$
    using the Dominated Convergence Theorem and recalling \Cref{equation in the sense of distributions}, one can easily check that 
    \begin{equation}\label{first variation}
        I'(\varepsilon)=\int_\Omega(|D(u+\varepsilon\xi)|^{p-2}D(u+\varepsilon\xi)\cdot D\xi+\gamma(u+\varepsilon\xi)^{\gamma-1}\xi)\,dx
    \end{equation}
    and
    \begin{align}\label{second variation}
        I''(\varepsilon)&=\int_{B_{3r}}|D(u+\varepsilon\xi)|^{p-2}|D\xi|^2\,dx\nonumber\\
        &\quad+\int_{B_{3r}}(p-2)|D(u+\varepsilon\xi)|^{p-4}(D(u+\varepsilon\xi)\cdot D\xi)^2\,dx\nonumber\\
        &\quad+\int_{B_{3r}}\gamma(\gamma-1)(u+\varepsilon\xi)^{\gamma-2}\xi^2\,dx.
    \end{align}
    Thus, $I(\varepsilon)$ attains its minimum at $\varepsilon=0$, and thanks to \eqref{first variation} and \Cref{equation in the sense of distributions}, one has
    $$
    I'(0)=\int_\Omega(|Du|^{p-2}Du\cdot D\xi+\gamma u^{\gamma-1}\xi)\,dx=0.
    $$
    Also, by the Monotone Convergence Theorem, $I'$ is continuous on $[0,1)$ and differentiable on $(0,1)$. Hence, there must exist $\varepsilon_k\to0$ such that $I''(\varepsilon_k)\ge0$. Thus, recalling the definition of $\xi$, employing \eqref{second variation} and \Cref{magic theorem}, we estimate
    \begin{equation*}
        \gamma(1-\gamma)\int_{B_{2r}}(u+\varepsilon_k)^{\gamma-2}\,dx\le\gamma(1-\gamma)\int_{B_{3r}}(u+\varepsilon_k\xi)^{\gamma-2}\xi^2\,dx\le C,
    \end{equation*}
    where $C>0$ is a constant depending only on $\|u\|_{L^\infty(B_1)}$, $\|D\xi\|_{L^\infty(B_1)}$, $p$, $\gamma$ and $n$. Using the Monotone Convergence Theorem one more time, we obtain $u^{\gamma-2}\in L^1(B_{2r})$.
    \vspace{0.2cm}

    \textsc{Step 2.} Next, we show that \eqref{new integrability} implies \eqref{step1}. 
    
    Recall that if $g\in L^1_{\loc}(B_{2r})$, and $s>0$, then (see, for example, \cite[2.10.19 (4)]{F69} or \cite[Theorem 5.8, see also Section 6.2]{M12})
    \begin{equation}\label{Federer69}
        \lim_{\rho\to0}\rho^{s-n}\int_{B_\rho(x_0)}g\,dx=0\,\,\,\text{for}\,\,\,\mathcal{H}^{n-s}\,\,\,\text{a.e.}\,\,\,x_0\in B_{2r}.
    \end{equation}
    If $x_0\in\partial\{u>0\}\cap B_r$, then by \eqref{growth}, one has $u<Cr^\alpha$, where $\alpha:=\frac{p}{p-\gamma}$ and $C>0$ depends only on $p$, $\gamma$ and $n$. Hence, for $\rho<\frac{r}{3}$, one has
    $$
    \int_{B_\rho(x_0)}u^{\gamma-2}\,dx>C\rho^{n-\alpha(2-\gamma)},
    $$
    that is
    $$
    \rho^{\alpha(2-\gamma)-n}\int_{B_\rho(x_0)}u^{\gamma-2}\,dx>C.
    $$
    Consequently, \eqref{Federer69} gives $\mathcal{H}^{n-\alpha(2-\gamma)}(\partial\{u>0\}\cap B_r)=0$. To finish the proof, observe that $\partial\{u>0\}\cap B_r=\{u=0\}\cap B_r$ since $|\{u=0\}\cap B_r|=0$. 
\end{proof}
\begin{corollary}\label{n-2 Hausdorff}
    If $u$ is a minimizer of \eqref{mimization functional}, then 
    $$
    \mathcal{H}^{n-\frac{2-\gamma}{p-\gamma}p}\left(\partial\{u>0\}\setminus\overline{\partial_{\reduced}\{u>0\}}\right)=0.
    $$
\end{corollary}
\begin{proof} 
    First, we observe that
    \begin{equation}\label{3.3}
    x\in\partial\{u>0\}\setminus\overline{\partial_{\reduced}\{u>0\}}\iff|\{u=0\}\cap B_r(x)|=0.
    \end{equation}
    Indeed, we know that locally $|\partial\{u>0\}|=0$ (since the free boundary has locally finite $(n-1)$-dimensional Hausdorff measure, \cite[Theorem 6.1]{ATU25}). This means that $|\{u=0\}|=|\interior\{u=0\}|$. Observe that if $x\in\partial\{u>0\}$, then for some $r>0$,
    \begin{equation}\label{3.2}
    x\in\partial\{u>0\}\setminus\partial(\interior\{u=0\}) \iff |\{u=0\}\cap B_r(x)|=0.
    \end{equation}
    From the general theory of sets with locally finite perimeter (see, for example, \cite[Proposition 12.19, (15.3) and (16.27)]{M12}), one has $x\in \partial\{u>0\}\setminus\overline{\partial_{\reduced}\{u>0\}}$ if and only if for some $r>0$ either
    \begin{equation}\label{(a)}
    |\{u=0\}\cap B_r(x)|=0
    \end{equation}
    or
    \begin{equation}\label{(b)}
    |\{u>0\}\cap B_r(x)|=0.
    \end{equation}
    But \eqref{positive density} excludes \eqref{(b)}. Thus, only \eqref{(a)} is possible, and \eqref{3.3} follows. 
    
    Note that \eqref{3.2} yields
    $$
    \overline{\partial_{\reduced}\{u>0\}}=\partial\left(\interior\{u=0\}\cap\Omega\right).
    $$
    We then cover $\partial\{u>0\}\setminus\overline{\partial_{\reduced}\{u>0\}}$ with balls $\left\{B_{r(x)}(x)\right\}$ satisfying \eqref{(a)}. By \eqref{step1}, for each $x\in\partial\{u>0\}\setminus\overline{\partial_{\reduced}\{u>0\}}$, one has
    $$
    \mathcal{H}^{n-\frac{2-\gamma}{p-\gamma}p}\left(\{u=0\}\cap B_{\frac{r(x)}{3}}(x)\right)=0,
    $$
    and selecting a countable subcovering from $\left\{B_{r(x)}(x)\right\}$ gives the desired result.
\end{proof}

\section{Rectifiability of the free boundary}\label{s5}
We now have all the ingredients to prove \eqref{reduced}, extending \cite[Theorem 3.2]{Ph83} for $p=2$ to the nonlinear setting with $p\ge2$.
\begin{theorem}\label{Hausdorff reduced}
If $u$ is a minimizer of \eqref{mimization functional}, then there exists $r_0>0$, depending only on $p$, $\gamma$ and $n$, such that for any $r\le r_0$ and $x_0\in\partial\{u>0\}\cap B_r(x_0)$, one has
$$
    \mathcal{H}^{n-1}\left((\partial\{u>0\}\setminus\partial_{\reduced}\{u>0\})\cap B_r(x_0)\right)=0.
$$
\end{theorem}
\begin{proof}
Using Federer's decomposition for sets with locally finite perimeter, \cite[Section 4.5.6, page 478]{F69}, we can decompose the free boundary $\partial\{u>0\}$ as follows:
$$
\partial\{u>0\}=\partial_{\reduced}\{u>0\}\cup A\cup B,
$$
where $\H(B)=0$ and $x_0\in A$ if and only if either 
\begin{equation}\label{zero density in u>0}
    \limsup_{r\to0}\frac{|\{u>0\}\cap B_r(x_0)|}{|B_r(x_0)|}=0
\end{equation}
or
\begin{equation}\label{zero density in u=0}
    \limsup_{r\to0}\frac{|\{u=0\}\cap B_r(x_0)|}{|B_r(x_0)|}=0
\end{equation}
We aim to show that $\H(A)=0$. As before, \eqref{positive density} excludes \eqref{zero density in u>0}. Thus, $x_0\in A$ if and only if \eqref{zero density in u=0} holds. Below, we show that it is not possible.

\textsc{Step 1.} Note that
\begin{equation*}\label{integrability}
    u^{-\frac{1}{\alpha}}\left|\frac{p}{p-1}-\frac{|Du|^p}{u^\gamma}\right|\in L^{1}\left(\{u>0\}\cap B_r(x_0)\right),
\end{equation*}
where $\alpha:=\frac{p}{p-\gamma}$. Indeed, it follows from \Cref{integrability lemma}, as by direct calculation, using \eqref{EL}, in $\{u>0\}$ one has
\begin{equation*}
    \Delta_p(u^{\frac{1}{\alpha}})=\frac{\gamma}{\alpha^{p-1}}\frac{p-1}{p}u^{-\frac{1}{\alpha}}\left[\frac{p}{p-1}-\frac{|Du|^p}{u^\gamma}\right].
\end{equation*}
Recalling \eqref{Federer69}, for $\H$ a.e. $x_0\in A$, we have
\begin{equation}\label{4.7}
    \lim_{r\to0}r^{1-n}\int_{\{u>0\}\cap B_r(x_0)} u^{-\frac{1}{\alpha}}\left|\frac{p}{p-1}-\frac{|Du|^p}{u^\gamma}\right|\,dx=0.
\end{equation}
	Set
	$$
	u_k(y):=\frac{u(x_0+r_ky)}{r_k^\alpha},\,\,\,y\in B_1.
	$$
    By \eqref{local regularity}, $u_k$ is uniformly bounded in $B_{\frac{1}{2}}$. Thus, there exists a function $u_0$ such that, up to a subsequence, 
    \begin{equation*}\label{strong convergence}
        u_k\to u_0\,\,\,\text{in}\,\,\,C^1(B_{\frac{1}{2}}). 
    \end{equation*}
    Furthermore, since $u_k$ is a local minimizer in $B_{\frac{1}{2}}$, so is $u_0$. Moreover,
    \begin{align*}
        &r^{1-n}\int_{\{u>0\}\cap B_r(x_0)} u^{-\frac{1}{\alpha}}\left|\frac{p}{p-1}-\frac{|Du|^p}{u^\gamma}\right|\,dx\\
        &=\int_{\{u_k>0\}\cap B_1} u_k^{-\frac{1}{\alpha}}\left|\frac{p}{p-1}-\frac{|Du_k|^p}{u_k^\gamma}\right|\,dx.
    \end{align*}
    Therefore, using Fatou's lemma and \eqref{4.7}, we get
    \begin{align*}
        0&=\lim_{k\to\infty}\int_{\{u_k>0\}\cap B_1} u_k^{-\frac{1}{\alpha}}\left|\frac{p}{p-1}-\frac{|Du_k|^p}{u_k^\gamma}\right|\,dx\\
        &\ge\int_{\{u_0>0\}\cap B_{\frac{1}{2}}} u_0^{-\frac{1}{\alpha}}\left|\frac{p}{p-1}-\frac{|Du_0|^p}{u_0^\gamma}\right|\,dx
    \end{align*}
    and hence,
    \begin{equation}\label{3.9}
        \frac{p}{p-1}=\frac{|Du_0|^p}{u_0^\gamma}\,\,\,\text{on}\,\,\,\{u_0>0\}\cap B_{\frac{1}{2}}.
    \end{equation}
    \textsc{Step 2.} Next, we show that
    \begin{equation}\label{blow-up limit zero set zero measure}
        |\{u_0=0\}\cap B_{\frac{1}{4}}|=0.
    \end{equation} 
    Indeed, if $|\{u_0=0\}\cap B_{\frac{1}{4}}|>0$, then $u_0=0$ on a ball $B_\rho(z)\subset B_{\frac{1}{4}}$. Thus,
    \begin{equation}\label{converges to zero}
        u_k\to0\,\,\,\text{on}\,\,\,B_\rho(z).
    \end{equation}    
    As $x_0\in A$, \eqref{zero density in u=0} yields, up to a subsequence,
    $$
    |\{u_k=0\}\cap B_{\frac{1}{2}}|\to0,\,\,\,\text{as}\,\,\,k\to\infty.
    $$
    Hence, for $k$ large enough, there exist $y_k\in B_{\frac{\rho}{2}}(z)$ such that $u_k(y_k)>0$. Since $u_k$ is a local minimizer, by \eqref{nondegeneracy} one has
    $$
    \sup_{\partial B_{\frac{\rho}{4}}}u_k\ge c\rho^\alpha,
    $$
    which contradicts \eqref{converges to zero}. Hence, \eqref{blow-up limit zero set zero measure} holds.

    \textsc{Step 3.} In the final step, we show that \eqref{blow-up limit zero set zero measure} implies that $u_0$ is one-dimensional, leading to a contradiction. Indeed, since $u_0$ is a local minimizer, it satisfies \eqref{EL}. Direct calculation then reveals that in $\{u_0>0\}$ one has
    \begin{equation}\label{v is p-harmonic}
        v\Delta_pv=\frac{\gamma}{\alpha^{p-1}}-(\alpha-1)(p-1)|Dv|^p,
    \end{equation}
    where
    $$
    v:=u_0^{\frac{1}{\alpha}}.
    $$
    On the other hand, using \eqref{3.9}, one gets
    \begin{equation}\label{gradient is constant}
        |Dv|^p=\frac{p}{p-1}\alpha^{-p}:=|a|^p,
    \end{equation}
    therefore, \eqref{v is p-harmonic} yields in $\{u_0>0\}$
    $$
    v\Delta_pv=\alpha^{-p}(\gamma\alpha-p(\alpha-1))=0,
    $$
    and since $\{v>0\}=\{u_0>0\}$, we conclude that $v$ is $p-$harmonic in its positivity set. As by \eqref{gradient is constant} it also has a constant magnitude of gradient, then $v$ is, in fact, harmonic in its positivity set. A harmonic function with a constant magnitude of the gradient is affine. Indeed, since
    $$
    0=\Delta(|Dv|^2)=2|D^2v|^2,
    $$
    then $D^2v\equiv0$. Thus,
    $$
    v(x)=a\cdot x+b.
    $$
    Since $v$ is nonnegative and vanishes at the origin, then $v(x)=(a\cdot x)^+$. Thus, $v(x)=|a|(\nu\cdot x)^+$, where $\nu:=a/|a|$ is a unit vector. Rotating the coordinates, we get 
    $$
    v(x)=|a|x_n^+\,\,\,\text{on}\,\,\,B_{\frac{1}{4}}.
    $$
    Therefore,
    $$
    \{u_0=0\}=\{x_n=0\}.
    $$
    On the other hand, by \eqref{blow-up limit zero set zero measure} and \Cref{Hausdorff local} (note that $\frac{2-\gamma}{p-\gamma}p>1$),  one should have
    $$
    \infty=\mathcal{H}^{n-\frac{2-\gamma}{p-\gamma}p}\left(\{x_n=0\}\cap B_{\frac{1}{4}}\right)=\mathcal{H}^{n-\frac{2-\gamma}{p-\gamma}p}\left(\{u_0=0\}\cap B_{\frac{1}{4}}\right)=0,
    $$
    a contradiction.
\end{proof}

\bigskip

{\small \noindent{\bf Acknowledgments.}} {\footnotesize This publication is based upon work supported by the King Abdullah University of Science and Technology (KAUST) under Award No. ORFS-CRG12-2024-6430. The author thanks Diego Marcon (UFRGS, Brazil) for helpful discussions at an early stage of this project.}

\end{document}